\documentclass[preprint,12pt]{elsarticle}
\usepackage[top=3cm, bottom=3cm, left=2.2cm, right=2.2cm]{geometry}
\usepackage{amssymb}    % 数学符号
\usepackage{amsmath}    % 基础公式环境
\usepackage{amsthm}     % 定理、引理、定义环境
\usepackage{mathtools}  % 复杂公式扩展（分段函数、积分等）
\usepackage{cases}      % 优化分段函数排版
\usepackage{hyperref}   % 可点击交叉引用
\usepackage{cleveref}   % 智能引用（自动带环境类型：定理/引理）

\newtheorem{theorem}{Theorem}[section]    % 定理：1.1, 2.1...
\newtheorem{lemma}[theorem]{Lemma}        % 引理：与定理共享编号
\newtheorem{definition}[theorem]{Definition}  % 定义：与定理共享编号
\newtheorem{proposition}[theorem]{Proposition} % 命题：与定理共享编号

\newcommand{\R}{{\mathbb R}}
\journal{Journal de Mathématiques Pures et Appliquées}

\begin{document}

\begin{frontmatter}

%% 标题（填写你的论文标题）
\title{Critical threshold for a two-species chemotaxis system with the energy critical exponent}

%% 作者（多个作者用 \and 分隔）
\author{Shen Bian \corref{cor1}}
\cortext[cor1]{Corresponding author. E-mail: bianshen66@163.com} % 通讯作者

%% 作者单位（多个单位用 label 区分，按需添加）
\affiliation{organization={Department of Mathematics, Beijing University of Chemical Technology},
            addressline={No. 15 East North Third Ring Road, Chaoyang District},
            city={Beijing},
            postcode={100029},
            country={China}}

%% 摘要（必填，控制在 150-200 词，概括研究目的、方法、结果）
\begin{abstract}
We consider a two-species chemotaxis model in $\R^d(d \ge 3)$ featuring nonlinear porous medium-type diffusion and nonlocal attractive power-law interaction. Here, the nonlinear diffusion is chosen to be $1/m_1+1/m_2=(d+2)/d$ in such a way that the associated free energy is conformal invariant, and there are radially symmetric, non-increasing and non-compactly supported stationary solutions $(U_s(x),V_s(x))$. We analyze the conditions on initial data $(u_0,v_0)$ under which attractive forces dominate over diffusion, and further classify the global existence and finite time blow-up of dynamical solutions by virtue of these stationary solutions. Specifically, the solution $(u,v)(x,t)$ exists globally in time if the initial data satisfy $\|u_0\|_{L^{m_1}(\R^d)}<\|U_s\|_{L^{m_1}(\R^d)}$ and $\|v_0\|_{L^{m_2}(\R^d)}<\|V_s\|_{L^{m_2}(\R^d)}$. In contrast, there are blowing-up solutions when $\|u_0\|_{L^{m_1}(\R^d)}>\|U_s\|_{L^{m_1}(\R^d)}$ and $\|v_0\|_{L^{m_2}(\R^d)}>\|V_s\|_{L^{m_2}(\R^d)}$.
\end{abstract}

%%Research highlights
\iffalse
\begin{highlights}
\item Conformal invariant critical exponent $1/m_1+1/m_2=(d+2)/d$ for free energy.
\item Stationary solutions $(U_s,V_s)$: radially symmetric, non-increasing and non-compactly supported, dynamical thresholds.
\item Sharp dichotomy: global existence or finite-time blow-up via stationary $L^{m_1}/L^{m_2}$ norms.
\item Critical balancing role of $L^{m_1}/L^{m_2}$ between diffusion and attraction).
\end{highlights}
\fi
%% 关键词（3-5 个，数学方向相关，用 \sep 分隔）

\end{frontmatter}

%% 正文开始
\section{Introduction}
In ecological systems, the interplay between multiple species is often intricately linked to their chemotactic behaviors, where organisms respond to chemical signals in their environment to navigate, interact, or survive. For two-species communities, such dynamics are particularly pronounced: each species may secrete specific signaling molecules, while simultaneously exhibiting directed movement toward the signals produced by themselves or the other species. This coupling of interspecific interactions and chemotactic response forms the core of two-species chemotaxis models. 

In numerous symbiotic and cooperative ecosystems, a key phenomenon emerges: two species mutually secrete chemical signals and exhibit positive directional movement toward one another. For instance, in ecology, pollinating insects are drawn to flowering plants in response to floral scent signals secreted by the plants; meanwhile, the insects release cuticular pheromones that in turn attract more conspecifics or other pollinators. Through bidirectional chemotaxis, coordination between plant pollen dispersal and insect food acquisition is achieved. In oncology, cancer-associated fibroblasts (CAFs) secrete chemokines to attract T cells to aggregate in tumor regions, and activated T cells also secrete signals that reciprocally attract CAFs. This bidirectional chemotactic dynamic influences the balance between tumor immune killing and immune evasion. In microbiology, nitrogen-fixing bacteria in soil secrete organic acids to attract phosphate-solubilizing bacteria, which in turn release amino acids that attract nitrogen-fixing bacteria. By virtue of bidirectional chemotaxis, these two types of bacteria form symbiotic aggregated patches, enabling the efficient completion of nitrogen and phosphorus nutrient cycling in soil.

The core commonalities of these phenomena are as follows: first, the two species each secrete chemical attractants, facilitating mutual directional chemotaxis; second, the diffusion behavior of the species depends on their own density—diffusion weakens at low density to avoid extinction (degenerate diffusion). To quantitatively describe such dynamics, the following system is proposed
\begin{align}\label{uvsystem}
\left\{
  \begin{array}{ll}
    u_t=\Delta u^{m_1}-\nabla \cdot (u \nabla c), & x \in \R^d, t \ge 0, \\
    -\Delta c=v,  & x \in \R^d, t \ge 0, \\
    v_t=\Delta v^{m_2}-\nabla \cdot (v \nabla z), & x \in \R^d, t \ge 0, \\
    -\Delta z=u, & x \in \R^d, t \ge 0, \\
    (u,v)(x,0)=(u_0,v_0)(x) \ge 0, & x \in \R^d.  
  \end{array}
\right.
\end{align}
Here, $u(x,t)$ and $v(x,t)$ represent the densities of two interacting agent types. The diffusion exponents $m_1>1$ and $m_2>1$ determine the diffusion capacity of the agents, which depends on their densities. The bidirectional chemotaxis is captured by the terms $-\nabla \cdot (u\nabla c)$ and $-\nabla \cdot (v \nabla z)$, reflecting mutual attraction. The chemical concentrations $c(x,t)$ and $z(x,t)$, emitted by $v$ and $u$ respectively, are given by the nonlocal expressions
\begin{align}
c=c_d \int_{\R^d} \frac{v(y)}{|x-y|^{d-2}} dy,\quad z=c_d \int_{\R^d} \frac{u(y)}{|x-y|^{d-2}} dy,
\end{align}
where the constant is defined by
\begin{align}
c_d=\frac{1}{d(d-2)\alpha_d} \quad \text{with} \quad \alpha_d = \frac{\pi^{d/2}}{\Gamma(d/2+1)}.
\end{align}
The system \eqref{uvsystem} admits several a priori estimates. Specifically, its solutions remain nonnegative, and mass conservation holds: 
\begin{align}
M_1:=\int_{\R^d} u(x,t) dx=\int_{\R^d} u_0(x) dx,\quad M_2:=\int_{\R^d} v(x,t) dx=\int_{\R^d} v_0(x) dx.
\end{align}
Throughout this paper, the initial data are assumed to satisfy
\begin{align}\label{initialdata}
(u_0,v_0) \in L_+^1 \cap L^\infty(\R^d),\quad \int_{\R^d} (|x|^2 u_0(x)+|x|^2 v_0(x)) dx <\infty.
\end{align}
The competition between diffusion and nonlocal aggregation is represented by the free energy
\begin{align}\label{Fuv}
F(u,v)=\frac{1}{m_1-1} \int_{\R^d} u^{m_1}dx + \frac{1}{m_2-1} \int_{\R^d} v^{m_2}dx-c_d \iint_{\R^d \times \R^d} \frac{u(x)v(y)}{|x-y|^{d-2}}dxdy.
\end{align}
This functional satisfies a dissipation principle:
\begin{align*}
\frac{d F(u,v)}{dt}=-\int_{\R^d} u \left|\frac{m_1}{m_1-1} \nabla u^{m_1-1}-\nabla c\right|^2 dx-\int_{\R^d} v \left|\frac{m_2}{m_2-1} \nabla v^{m_2-1}-\nabla z   \right|^2 dx \le 0.
\end{align*}
Note that system \eqref{uvsystem} possesses two scaling invariances. Indeed, if $u$ and $v$ solve \eqref{uvsystem}, then the scaled functions 
\begin{align}\label{uvscale}
\left\{
  \begin{array}{ll}
   u_\lambda(x,t)=\lambda^{\frac{2m_2}{m_1+m_2-m_1 m_2}} u(\lambda x, \lambda^{\frac{2m_1}{m_1+m_2-m_1 m_2}} t), \\
v_\lambda(x,t)=\lambda^{\frac{2m_1}{m_1+m_2-m_1 m_2}} v(\lambda x, \lambda^{\frac{2m_2}{m_1+m_2-m_1 m_2}} t) 
  \end{array}
\right.
\end{align}
also solve \eqref{uvsystem}. These scalings preserve the $L^{p_1}$ and $L^{p_2}$ norms: 
\begin{align}
\| u_{\lambda} \|_{L^{p_1}(\R^d)}=\|u\|_{L^{p_1}(\R^d)},\quad \| v_{\lambda} \|_{L^{p_2}(\R^d)}=\|v\|_{L^{p_2}(\R^d)},
\end{align}
where
\begin{align}\label{p1p2}
p_1=\frac{(m_1+m_2-m_1 m_2)d }{2 m_2},\quad p_2=\frac{(m_1+m_2-m_1 m_2)d }{2 m_1}.
\end{align}
The above scalings become mass invariant for both $u$ and $v$ if $m_1=m_2=2-2/d$. When $m_1 m_2+2m_1/d=m_1+m_2$, then $p_2=1$ and mass conservation holds for $v$, whereas $p_1=1$ and mass conservation holds for $u$ when $m_1 m_2+2m_2/d=m_1+m_2$. 

It is worth noting that single-species chemotaxis systems have been intensively studied in recent years. For $m_1=m_2>1$ and $d \ge 2$, the global existence and blow-up of dynamic solutions have been analyzed in \cite{BJ09,BCM08,bdp06,JL92,kimyao,bp07,su06,suku06}. Additionally, significant efforts have focused on investigating steady states, as documented in \cite{CCH17,CCH21,CCV15,CHVY19,CHM18,CGH20,DYY22}. For two-species systems, \cite{JK22} established results on global existence and finite-time blow-up for $m_1>1$, $m_2>1$, and $d \ge 3$. Specifically, for $m_1=m_2=2-2/d$, there exists a number $M_c>0$ such that solutions exist globally if $M_1 M_2 <M_c^2$ and blow up in finite time if $\frac{M_1 M_2}{M_1^{2-2/d}+M_2^{2-2/d}}>\frac{M_c^{2/d}}{2}.$ For $m_1 m_2+2m_1/d<m_1+m_2$ and $m_1 m_2+2m_2/d<m_1+m_2$, one can construct large initial data ensuring blow-up in finite time. When $m_1=m_2=1$ and $d=2,$ \cite{HW22} showed that solutions exist globally if $M_1 M_2 <4 \pi (M_1+M_2)$ and blow-up occurs if $M_1 M_2 >4 \pi (M_1+M_2)$. For multi-species populations, \cite{DG22,DG19} proved global existence and infinite-time aggregation in subcritical and critical cases. For multi-species populations in $\R^2$, \cite{DG22,DG19} proved global existence and infinite-time aggregation in subcritical and critical cases. 

Our main goal in this paper is to classify the global existence and finite-time blow-up of solutions based on the initial data. The main tool for analyzing \eqref{uvsystem} is the stationary system, which can be interpreted in the distributional sense as
\begin{align}\label{star0}
\left\{
  \begin{array}{ll}
    \frac{m_1}{m_1-1} U_s^{m_1-1}-C_s=\overline{C}_1, \text{ in } \Omega_1, \\ \
  \frac{m_2}{m_2-1} V_s^{m_2-1}-Z_s=\overline{C}_2, \text{ in } \Omega_2, \\ \
  U_s=0 \text{ in } \R^d\setminus\Omega_1,\quad U_s>0 \text{ in } \Omega_1, \\ \
  V_s=0 \text{ in } \R^d\setminus\Omega_2,\quad V_s>0 \text{ in } \Omega_2, \\ \
 C_s=c_d \int_{\R^d} \frac{V_s(y)}{|x-y|^{d-2}}dy, \quad Z_s=c_d \int_{\R^d} \frac{U_s(y)}{|x-y|^{d-2}}dy.
  \end{array}
\right.
\end{align}
Here, $\overline{C}_1$ and $\overline{C}_2$ are two constant chemical potentials, and $\Omega_1=\{ x\in \R^d\,| U_s>0 \}, \Omega_2=\{ x\in \R^d\,| V_s>0 \}.$ Letting $\phi_1=\frac{m_1-1}{m_1} (C_s+\overline{C}_1)$ and $\phi_2=\frac{m_2-1}{m_2} (Z_s+\overline{C}_2)$, the stationary system reduces to 
\begin{align}\label{star}
\left\{
  \begin{array}{ll}
    -\Delta \phi_1=\frac{m_1-1}{m_1} \phi_2^{\frac{1}{m_2-1}}, &\text{ in } \Omega_2, \\
  -\Delta \phi_2=\frac{m_2-1}{m_2} \phi_1^{\frac{1}{m_1-1}}, &\text{ in } \Omega_1.
  \end{array}
\right.
\end{align}
System \eqref{star} is the well-known Hénon-Lane-Emden system, which has been extensively studied by many authors. The Hénon-Lane-Emden conjecture states that there is no positive classical solution to \eqref{star} in $\Omega_1=\Omega_2=\R^d$ if and only if $1/m_1+1/m_2<(d+2)/d$. This conjecture is known to be true for radial solutions in all dimensions \cite{EM96,SZ98}. For non-radial solutions, the conjecture is only fully answered when $d \le 4$ \cite{EM01,PQ07,SZ96,S09}. Some partial results have also been established for $d \ge 5$ \cite{CLL17,CLZ17,DQ22,JL14}. One can refer to \cite{PS12} for 
\begin{align*}
m_1,m_2 \ge \frac{2d}{d+2} \text{ and } (m_1,m_2) \neq \left( \frac{2d}{d+2},\frac{2d}{d+2} \right),
\end{align*}
and \cite{AY14} for 
\begin{align*}
\frac{2m_1(m_2-1)}{m_1+m_2-m_1m_2},\frac{2m_2(m_1-1)}{m_1+m_2-m_1m_2} \ge \frac{d-2}{2}
\end{align*}
and
\begin{align*}
\left(\frac{2m_1(m_2-1)}{m_1+m_2-m_1m_2},\frac{2m_2(m_1-1)}{m_1+m_2-m_1m_2} \right ) \neq \left(\frac{d-2}{2},\frac{d-2}{2} \right).
\end{align*}
Moreover, both Mitidieri \cite{EM96} and Serrin-Zou \cite{SZ98} constructed positive radial solutions for $1/m_1+1/m_2 \ge (d+2)/d$. 

In this work, we focus on the case $1/m_1+1/m_2=(d+2)/d.$ Several aspects motivate the consideration of this special case:
\begin{itemize}
\item Based on the above analysis, stationary solutions fall into two subcases depending on $1/m_1+1/m_2$. For $1/m_1+1/m_2<(d+2)/d$, all the radial solutions and some non-radial solutions are compactly supported. In contrast, all radial stationary solutions are not compactly supported for $1/m_1+1/m_2 \ge (d+2)/d$. Consequently, $1/m_1+1/m_2=(d+2)/d$ is the critical case separating compactly supported stationary solutions from those that are not.
\item The free energy $F(u,v)$ is invariant scaling when $1/m_1+1/m_2=(d+2)/d$. Indeed, substituting the scaling functions \eqref{uvscale} into the free energy $F(u,v)$, we obtain
    \begin{align}
     F(u_\lambda,v_\lambda)=\lambda^{\frac{(d+2)m_1m_2-d(m_1+m_2)}{m_1+m_2-m_1m_2}}F(u,v),
    \end{align}
    and the free energy is invariant when $1/m_1+1/m_2=(d+2)/d$. For this reason, this case corresponds to the energy-critical curve. We also infer from \eqref{uvscale} that $\|u_\lambda\|_{L^{m_1}(\R^d)}=\|u\|_{L^{m_1}(\R^d)}$ and $\|v_\lambda\|_{L^{m_2}(\R^d)}=\|v\|_{L^{m_2}(\R^d)}$.
\item From the expression of the free energy $F(u,v)$, the energy-critical curve $1/m_1+1/m_2=(d+2)/d$ serves as the equilibrium threshold between the ``tendency of particles toward free diffusion'' and the ``constraint from long-range attraction''. The non-trivial extremal solutions correspond to the stable states with the lowest energy.
\end{itemize}

In fact, for $1/m_1+1/m_2=(d+2)/d$, we can deduce that $\overline{C}_1=\overline{C}_2=0$ and $\Omega_1=\Omega_2=\R^d$ in \eqref{star0} (see Proposition \ref{prop1006}). Therefore, the stationary solution to \eqref{uvsystem} can be summarized as follows.

\begin{theorem}\label{steadymain}
Let $1/m_1+1/m_2=(d+2)/d.$ Assume $(U_s,V_s) \in L^{m_1}(\R^d) \times L^{m_2}(\R^d)$ is a pair of stationary solutions of \eqref{uvsystem}. Then $(U_s,V_s)$ are radially symmetric and monotonic decreasing up to translation, and satisfy
\begin{align}
\left\{
  \begin{array}{ll}
 \frac{m_1}{m_1-1} U_s^{m_1-1}=c_d \int_{\R^d} \frac{V_s(y)}{|x-y|^{d-2}} dy, &  x \in \R^d, \\
\frac{m_2}{m_2-1} V_s^{m_2-1}=c_d \int_{\R^d} \frac{U_s(y)}{|x-y|^{d-2}} dy, &  x \in \R^d.
  \end{array}
\right.
\end{align}
In particular, if $m_1=m_2=2d/(d+2)$, then $U_s=V_s$ and they both assume the form 
\begin{align}
c \left( \frac{\lambda}{\lambda^2+|x-x_0|^2}  \right)^{\frac{d+2}{2}}
\end{align}
with some constant $c=c(d)$, and for some $\lambda>0$ and $x_0 \in \R^d$.
\end{theorem}

Therefore, we arrive at
\begin{align}\label{UsVs}
\frac{m_1}{m_1-1} \int_{\R^d} U_s^{m_1} dx&=\int_{\R^d} C_s U_s dx=c_d \iint_{\R^d \times \R^d} \frac{U_s(x) V_s(y)}{|x-y|^{d-2}} dxdy \nonumber\\
&=\int_{\R^d} Z_s V_s dx=\frac{m_2}{m_2-1} \int_{\R^d} V_s^{m_2} dx. 
\end{align}
Furthermore, the free energy of stationary solutions satisfies
\begin{align}\label{FUsVs}
F(U_s,V_s)&=\frac{1}{m_1-1} \int_{\R^d} U_s^{m_1}dx + \frac{1}{m_2-1} \int_{\R^d} V_s^{m_2}dx-c_d \iint_{\R^d \times \R^d} \frac{U_s(x)V_s(y)}{|x-y|^{d-2}}dxdy \nonumber\\
&=\frac{1}{m_1-1} \int_{\R^d} U_s^{m_1}dx-\int_{\R^d} V_s^{m_2}dx \nonumber\\
&=\frac{m_1+m_2-m_1m_2}{(m_1-1)m_2}\int_{\R^d} U_s^{m_1}dx.
\end{align}

Before proceeding, we first state the definition of the weak solution $(u,v)$ that we will deal with throughout this paper.
\begin{definition}
Let $(u_0,v_0)$ be the initial data satisfying \eqref{initialdata} and $T \in (0,\infty].$ A weak solution $(u,v)$ to \eqref{uvsystem} is a pair of non-negative functions $(u,v) \in \left(L^\infty \left(0,T;L_+^1 \cap L^\infty(\R^d) \right) \right)^2$ such that for any $0<t<T$ and all test functions $\psi_1, \psi_2 \in C_0^\infty(\R^d)$, 
 \begin{align*}
 &\int_{\R^d} \psi_1 u(\cdot,t)dx-\int_{\R^d} \psi_1 u_0(x) dx =\int_0^t
 \int_{\R^d} u^{m_1} \Delta \psi_1 dx ds \nonumber \\
 & - \frac{c_{d}(d-2)}{2} \int_0^t \iint_{\R^d\times \R^d}  \frac{[\nabla
 \psi_1(x)-\nabla \psi_1(y)] \cdot (x-y)}{|x-y|^2} \frac{u(x,s)
 v(y,s)}{|x-y|^{d-2}} dxdy ds, \nonumber \\
 &\int_{\R^d} \psi_2 v(\cdot,t)dx-\int_{\R^d} \psi_2 v_0(x) dx =\int_0^t
 \int_{\R^d} v^{m_2} \Delta \psi_2 dx ds \nonumber \\
 & - \frac{c_{d}(d-2)}{2} \int_0^t \iint_{\R^d\times \R^d}  \frac{[\nabla
 \psi_2(x)-\nabla \psi_2(y)] \cdot (x-y)}{|x-y|^2} \frac{v(x,s)
 u(y,s)}{|x-y|^{d-2}} dxdy ds.
 \end{align*}
\end{definition}

We mention that for $1/m_1+1/m_2=(d+2)/d$, the system \eqref{uvsystem} preserves the $L^{m_1}$ norm of $u$ and the $L^{m_2}$ norm of $v$ under the scaling \eqref{uvscale}. Indeed, the $L^{m_1}$ and $L^{m_2}$ norms play a critical role in determining dynamical behaviors of solutions to \eqref{uvsystem}. Our main results can be summarized as follows. 
\begin{theorem}\label{main}
Let $d \ge 3$ and $1/m_1+1/m_2=(d+2)/d.$ Assume that $(U_s,V_s)$ is a pair of stationary solutions of \eqref{uvsystem}. Under assumption \eqref{initialdata}, we also suppose $F(u_0,v_0)<F(U_s,V_s)$.
\begin{enumerate}
  \item[\textbf{(i)}] If $\|u_0\|_{L^{m_1} (\R^d)}<\|U_s\|_{L^{m_1} (\R^d)}$ and $\|v_0\|_{L^{m_2} (\R^d)}<\|V_s\|_{L^{m_2} (\R^d)}$, then there exists a 
     global weak solution $(u,v)$ to \eqref{uvsystem} satisfying that for any $0<t<\infty,$
     \begin{align}
      \|\left(u(\cdot,t),v(\cdot,t)\right)\|_{L^1 \cap L^\infty(\R^d)} \le C\left( \|\left(u_0,v_0\right)\|_{L^1 \cap L^\infty(\R^d)} \right).
     \end{align}
  \item[\textbf{(ii)}] If $\|u_0\|_{L^{m_1} (\R^d)}>\|U_s\|_{L^{m_1} (\R^d)}$ and $\|v_0\|_{L^{m_2} (\R^d)}>\|V_s\|_{L^{m_2} (\R^d)}$, then the weak solution of \eqref{uvsystem} blows up in finite time $T$, in the sense that
     \begin{align}
       \displaystyle \limsup_{t \to T} \left(\|u(\cdot,t)\|_{L^\infty(\R^d)}+\|v(\cdot,t)\|_{L^\infty(\R^d)} \right)=\infty.
     \end{align}
\end{enumerate}
\end{theorem}

The remainder of this paper is organized as follows. Section \ref{sec2} derives an existence criterion that characterizes the maximal existence time of solutions to \eqref{uvsystem}. Section \ref{sec3} explores the stationary solutions, which satisfy a variational structure related to the Hardy-Littlewood-Sobolev (HLS) inequality. Section \ref{sec4} makes use of the extremal functions of the HLS inequality to prove the main theorem concerning the dichotomy: global existence for small initial data and finite-time blow-up for large initial data. Finally, Section \ref{sec5} concludes the paper.

In what follows, we denote by $C$ a generic constant (which may vary between lines) and by $C=C(\cdot,\cdots,\cdot)$ a constant depending only on the quantities appearing in parentheses. We also use the simplified notations:
\begin{align*}
\|\cdot\|_{r}:=\|\cdot\|_{L^r(\R^d)}, 1 \le r <\infty, \quad \int \cdot \,dx:=\int_{\R^d} \cdot \, dx,\quad \iint \cdot \,dx:=\iint_{\R^d\times \R^d} \cdot \, dx
\end{align*}

\section{Preliminaries} \label{sec2}

In this section, we first recall the Hardy-Littlewood-Sobolev inequality (HLS inequality) \cite{lieb202}, which will be used to analyze the dynamics of solutions to \eqref{uvsystem}.
\begin{lemma}[HLS inequality]
Let $q,r>1$, $d \ge 3$ and $0<\beta<d$ with $1/q+1/r+\beta/d=2.$ Assume $f \in L^q(\R^d)$ and $h \in L^{r}(\R^d)$. Then there exists a sharp constant $C(d,\beta,q)$, independent of $f$ and $h$, such that     
\begin{align}\label{fh}
\left| \iint \frac{f(x)h(y)}{|x-y|^{\beta}} dxdy   \right| \le C(d,\beta,q) \|f\|_q \|h\|_r.
\end{align}
The sharp constant satisfies
\begin{align}\label{cdbeta}
C(d,\beta,q) \le \frac{d}{d-\beta} \left( \frac{\pi^{d/2}}{\Gamma(d/2+1)} \right)^{\beta/d} \frac{1}{qr}\left( \left(\frac{\beta/d}{1-1/q}\right)^{\beta/d}+\left( \frac{\beta/d}{1-1/r} \right)^{\beta/d} \right).
\end{align}
If $q=r=2d/(2d-\beta),$ then
\begin{align}
C(d,\beta,q)=C(d,\beta)=\pi^{\beta/2} \frac{\Gamma\left( d/2-\beta/2 \right)}{\Gamma\left( d-\beta/2 \right)} \left( \frac{\Gamma(d/2)}{\Gamma(d)} \right)^{-1+\beta/d}.
\end{align}
In this case, equality holds in \eqref{fh} if and only if $h$ is a constant multiple of $f$ and
\begin{align}
f(x)=B \left( \gamma^2+|x-x_0|^2 \right)^{-(2d-\beta)/2}
\end{align}
for some $B>0$, $\gamma \in \R$ and $x_0 \in \R^d.$
\end{lemma}

Now, we consider the regularized problem 
\begin{align}\label{uveps}
\left\{
  \begin{array}{ll}
    \partial_t u_\varepsilon=\Delta u_\varepsilon^{m_1}-\nabla \cdot \left( u_\varepsilon \nabla c_\varepsilon \right),\quad & x \in \R^d,~t \ge 0, \\
-\Delta c_\varepsilon=J_\varepsilon \ast v_\varepsilon,\quad & x \in \R^d,~t \ge 0, \\ 
 \partial_t v_\varepsilon=\Delta v_\varepsilon^{m_2}-\nabla \cdot \left( v_\varepsilon \nabla z_\varepsilon \right),\quad & x \in \R^d,~t\ge 0, \\
-\Delta z_\varepsilon=J_\varepsilon \ast u_\varepsilon,\quad & x \in \R^d,~t \ge 0, \\ 
 (u_\varepsilon,v_\varepsilon)(x,0)=(u_{0\varepsilon},v_{0\varepsilon}) \ge 0,\quad & x\in \R^d,
  \end{array}
\right.
\end{align}
where $J_\varepsilon(x)=\frac{1}{\varepsilon
^{d}}J(\frac{x}{\varepsilon})$, $J(x)=\frac{1}{\alpha_d}
\left(1+|x|^2\right)^{-(d+2)/2}$ and $\int J_\varepsilon(x)
dx=1.$ A straightforward computation shows that $c_\varepsilon$ and $z_\varepsilon$ can be expressed as
\begin{align}
c_\varepsilon=R_\varepsilon \ast v_\varepsilon,\quad z_\varepsilon=R_\varepsilon \ast u_\varepsilon
\end{align}
with the regularized Riesz potential defined by
\begin{align}
R_\varepsilon(x)=c_{d} \frac{1}{\left( |x|^2+\varepsilon^2  \right)^{\frac{d-2}{2}}}.
\end{align}
Here, $(u_{0\varepsilon},v_{0\varepsilon})$ denotes a sequence of approximations to the initial data $(u_0,v_0)$. This sequence can be constructed so that there exists $\varepsilon_0>0$ such that for any $0<\varepsilon<\varepsilon_0$,
\begin{align}
\left\{
  \begin{array}{ll}
 (u_{0\varepsilon},v_{0\varepsilon}) \ge 0, \quad \|u_\varepsilon(x,0)\|_1=\|u_0\|_1,\quad \|v_\varepsilon(x,0)\|_1=\|v_0\|_1, \\[1mm]
 \int |x|^2 u_{0\varepsilon} dx \to \int |x|^2 u_0(x)dx, \, \int |x|^2 v_{0\varepsilon} dx \to \int |x|^2 v_0(x)dx, \text{  as  }\varepsilon \to 0, \\[1mm]
 (u_{0\varepsilon}(x),v_{0\varepsilon}(x)) \to (u_0(x),v_0(x)) \text{  in  } L^q(\R^d),\text{  for  } 1 \le q < \infty,\text{  as  } \varepsilon \to 0.
  \end{array}
\right.
\end{align}

The regularized problem \eqref{uveps} admits global in time smooth solutions for any $\varepsilon>0$. Moreover, this approximation has been proved to be convergent. More precisely, following the arguments in \cite[Lemma 4.8]{CHVY19} and \cite[Section 4]{suku06}, we assert that if
\begin{align}\label{Linfinity}
\|u_\varepsilon(\cdot,t)\|_{L^\infty(\R^d)}+\|v_\varepsilon(\cdot,t)\|_{L^\infty(\R^d)} <C_0
\end{align}
for all $0<t<\infty$ (where $C_0$ is a constant independent of $\varepsilon$), then there exists a subsequence $\varepsilon_n \to 0$ such that
\begin{align}
  \begin{array}{ll}
 (u_{\varepsilon_n},v_{\varepsilon_n}) \rightarrow (u,v) \text{  in  } L^r(0,T;L^r(\R^d)),\quad 1 \le r<\infty \label{conver1}
  \end{array}
\end{align}
and $(u,v)$ is a weak solution to \eqref{uvsystem} on $[0,T).$ 

According to the above analysis, a weak solution $(u,v)$ to \eqref{uvsystem} on $[0,T)$ exists when \eqref{Linfinity} is fulfilled. Our focus will therefore be on establishing the validity of the $L^\infty$-bound. As shown in the following lemma, which provides local existence and blow-up criteria, this $L^\infty$ bound can be inferred from the $L^r$ norm of $(u_\varepsilon,v_\varepsilon)$ for $r>\max \left( \frac{d(2-m_1)}{2},\frac{d(2-m_2)}{2} \right)$. This further characterizes the maximal existence time.

\begin{lemma}[Local existence and blow-up criteria]\label{ueps}
Let $1/m_1+1/m_2=(d+2)/d$ with $1<m_1<2-2/d$ and $1<m_2<2-2/d$. Under assumption \eqref{initialdata} on the initial data, there exists a pair of weak solutions $(u,v)$ to \eqref{uvsystem} with a maximal existence time $T_w \in (0,\infty]$. Moreover, if $T_w<\infty$, then 
\begin{align*}
\limsup\limits_{t \to T_w} \left(\|u(\cdot,t)\|_{L^\infty(\R^d)}+\|v(\cdot,t)\|_{L^\infty(\R^d)}\right) = \infty.
\end{align*}
\end{lemma}
\begin{proof}
The key to proving local existence lies in establishing the following a priori estimates. We first denote
\begin{align*}
p_0=\max \left( \frac{d(2-m_1)}{2},\frac{d(2-m_2)}{2} \right).
\end{align*}

{\it\textbf{Step 1}} ($L^r$-estimates for $r>p_0$) \quad We multiply equation $\eqref{uveps}_1$ by $ru_\varepsilon^{r-1}$ and $\eqref{uveps}_3$ by $rv_\varepsilon^{r-1}$ to obtain
\begin{align}\label{eq22}
& \frac{d}{dt} \int (u_\varepsilon^r+v_\varepsilon^r) dx+\frac{4m_1 r(r-1)}{(m_1+r-1)^2} \int \left| \nabla u_\varepsilon^{\frac{m_1+r-1}{2}} \right|^2 dx+\frac{4m_2 r(r-1)}{(m_2+r-1)^2} \int \left| \nabla v_\varepsilon^{\frac{m_2+r-1}{2}} \right|^2 dx  \nonumber \\
=& (r-1) \int (u_\varepsilon^r J_\varepsilon \ast v_\varepsilon +v_\varepsilon^r J_\varepsilon \ast u_\varepsilon) dx \nonumber\\
\le & (r-1) \int (u_\varepsilon^r J_\varepsilon \ast u_\varepsilon +v_\varepsilon^r J_\varepsilon \ast v_\varepsilon) dx \nonumber\\
\le & (r-1) \left( \|u_\varepsilon^r \|_{(r+1)/r} \|J_\varepsilon \ast u_\varepsilon\|_{r+1}+\|v_\varepsilon^r \|_{(r+1)/r} \|J_\varepsilon \ast v_\varepsilon\|_{r+1} \right) \nonumber \\
\le & (r-1) \left( \|u_\varepsilon^r \|_{(r+1)/r} \|u_\varepsilon\|_{r+1}+\|v_\varepsilon^r \|_{(r+1)/r} \|v_\varepsilon\|_{r+1} \right) \nonumber \\
= & (r-1) \int (u_\varepsilon^{r+1}+v_\varepsilon^{r+1}) dx,
\end{align}
where we have used H\"{o}lder's inequality. For $r>p_0$, applying the Gagliardo-Nirenberg-Sobolev (GNS) inequality with $1<m_1<2-2/d$ and $1<m_2<2-2/d$, we find
\begin{align}
\int (u_\varepsilon^{r+1}+v_\varepsilon^{r+1}) dx \le C \left\|\nabla u_\varepsilon^{\frac{m_1+r-1}{2}} \right\|_2^{\alpha_1} \|u_\varepsilon\|_r^{\beta_1} +C \left\|\nabla v_\varepsilon^{\frac{m_2+r-1}{2}} \right\|_2^{\alpha_2} \|v_\varepsilon\|_r^{\beta_2},
\end{align}
where 
\begin{align*}
 \beta_1=r+1-\frac{r \alpha_1}{2}, \quad \beta_2=r+1-\frac{r \alpha_2}{2}
\end{align*}
and
\begin{align*}
\alpha_1=\frac{2}{1+2(r-d(2-m_1)/2)/d}<2, \quad \alpha_2=\frac{2}{1+2(r-d(2-m_2)/2)/d}<2
\end{align*}
ensured by $r>p_0$. Using Young's inequality, we further obtain
\begin{align}\label{eq33}
\int (u_\varepsilon^{r+1}+v_\varepsilon^{r+1}) dx \le & \frac{2m_1 r}{(m_1+r-1)^2}  \int \left|\nabla u_\varepsilon^{\frac{m_1+r-1}{2}}\right|^2 dx+\frac{2m_2 r}{(m_2+r-1)^2} \int \left|\nabla v_\varepsilon^{\frac{m_2+r-1}{2}}\right|^2 dx \nonumber \\
&+C(r,d) \left( \|u_\varepsilon\|_r^{r \delta_1} + \|v_\varepsilon\|_r^{r \delta_2} \right),
\end{align}
where 
\begin{align*}
\delta_1=1+\frac{1}{r-d(2-m_1)/2}>1, \quad \delta_2=1+\frac{1}{r-d(2-m_2)/2}>1.
\end{align*}
Substituting \eqref{eq33} into \eqref{eq22}, we arrive at                                                                                                                                                                                          
\begin{align}\label{eq44}                                            
&\frac{d}{dt} \int \left(u_\varepsilon^r+v_\varepsilon^r \right) dx+\frac{2m_1 r}{(m_1+r-1)^2}  \int \left|\nabla u_\varepsilon^{\frac{m_1+r-1}{2}}\right|^2 dx+\frac{2m_2 r}{(m_2+r-1)^2} \int \left|\nabla v_\varepsilon^{\frac{m_2+r-1}{2}}\right|^2 dx \nonumber \\
\le & C(r,d) \left( \|u_\varepsilon\|_r^{r \delta_1} + \|v_\varepsilon\|_r^{r \delta_2} \right). 
\end{align}
Consequently, the $L^r$-norm for $r>p_0$ is bounded locally in time:
\begin{align}\label{eq666}
\|u_\varepsilon\|_r^{r} + \|v_\varepsilon\|_r^{r} \le \frac{C(r,d)}{(T_r-t)^{r-p_0}},\quad T_r=\left( \|u_{0\varepsilon}\|_r^{r} + \|v_{0\varepsilon}\|_r^{r} \right)^{-\frac{1}{r-p_0}}. 
\end{align}
Returning to \eqref{eq44}, we also deduce that
\begin{align}\label{eq55}
\left\|\nabla u_\varepsilon^{\frac{m_1+r-1}{2}} \right\|_{L^2(0,T_r;L^2(\R^d))}+\left\|\nabla v_\varepsilon^{\frac{m_2+r-1}{2}} \right\|_{L^2(0,T_r;L^2(\R^d))} \le C\left( \|u_{0\varepsilon}\|_r, \|v_{0\varepsilon}\|_r \right).
\end{align}

{\it\textbf{Step 2}} ($L^\infty$-estimates and local existence)\quad As a direct consequence of Step 1, we apply the Moser iteration method to yield 
\begin{align}\label{eq66}
\displaystyle \sup_{0<t<T_r} \left(\|u_\varepsilon(\cdot,t)\|_{L^\infty(\R^d)} +\|v_\varepsilon(\cdot,t)\|_{L^\infty(\R^d)}\right) 
\le  C\left( \|u_{0\varepsilon}\|_1,\|u_{0\varepsilon}\|_{L^\infty(\R^d)}, \|v_{0\varepsilon}\|_1,\|v_{0\varepsilon}\|_{L^\infty(\R^d)}\right),
\end{align}
following the proof of \cite[Theorem 4.2]{BL14} verbatim. Armed with the regularities \eqref{eq666}, \eqref{eq55}, and \eqref{eq66}, we further derive
\begin{align*}
\|\nabla u_\varepsilon \|_{L^2(0,T_r;L^2(\R^d))} + \|\nabla v_\varepsilon \|_{L^2(0,T_r;L^2(\R^d))}   \le C, \\
\|u_\varepsilon \nabla c_\varepsilon \|_{L^\infty(0,T_r;L^\infty(\R^d))}+ \|v_\varepsilon \nabla z_\varepsilon \|_{L^\infty(0,T_r;L^\infty(\R^d))} \le C,
\end{align*}
where the constants $C$ depend only on $\|u_{0\varepsilon}\|_1$, $\|v_{0\varepsilon}\|_1$, $\|u_{0\varepsilon}\|_{L^\infty(\R^d)}$ and $\|v_{0\varepsilon}\|_{L^\infty(\R^d)}$. These a priori bounds hold uniformly in $\varepsilon$. Thus, we can pass to the limits $u_\varepsilon \to u$ and $v_\varepsilon \to v$ as $\varepsilon \to 0$ (without relabeling), due to the time compactness ensured by the Lions-Aubin lemma. In the limit, we obtain the local existence of the weak solution to \eqref{uvsystem}. Finally, in light of the proof of \cite[Theorem 2.4]{BJ09}, we characterize the maximal existence time of the weak solution and complete the proof. 
\end{proof}

\section{Steady states} \label{sec3}

This section is devoted to the analysis of stationary solutions $(U_s,V_s)$ to \eqref{uvsystem}. Section \ref{sec31} establishes general properties of these solutions. Specifically, Lemma \ref{31} provides a Pohozaev-Rellich type identity. Leveraging this identity, Proposition \ref{prop1006} determines the constant chemical potentials on the supports of $U_s$ and $V_s$. In Section \ref{sec32}, we explore the connection between the Hardy-Littlewood-Sobolev extremals and the steady states $(U_s,V_s)$, thus completing the proof of Theorem \ref{steadymain}. This link is crucial for establishing the global existence and finite-time blow-up of weak solutions in \eqref{uvsystem}.

\subsection{Steady states: proof of Theorem \ref{steadymain}}\label{sec31}
The stationary system for \eqref{uvsystem} is considered in the sense of distributions
\begin{align}\label{eq1}
\left\{
  \begin{array}{ll}
    \Delta U_s^{m_1}(x)-\nabla \cdot (U_s(x) \nabla C_s(x))=0, & x \in \R^d, \\
    \Delta V_s^{m_2}(x)-\nabla \cdot (V_s(x) \nabla Z_s(x))=0, & x \in \R^d,
  \end{array}
\right.
\end{align}
where the potentials $C_s(x)$ and $Z_s(x)$ are given by
\begin{align}
C_s=c_d \int \frac{V_s(y)}{|x-y|^{d-2}}dy,\quad Z_s=c_d \int \frac{U_s(y)}{|x-y|^{d-2}}dy.
\end{align}
We begin by deriving a Pohozaev-type identity to characterize the critical points of the free energy $F(u,v)$ associated with the steady states.

\begin{lemma}[A Pohozaev identity for stationary solutions]\label{31}
Assume that $U_s \in L^{m_1}(\R^d)$ and $V_s \in L^{m_2}(\R^d)$ satisfy \eqref{eq1} in the sense of distributions. Then the following identity holds:
\begin{align}\label{eq0}
2d \int U_s^{m_1} dx+2d \int V_s^{m_2} dx=2 c_d (d-2) \iint \frac{U_s(x) V_s(y)}{|x-y|^{d-2}} dxdy.
\end{align}
\end{lemma}
\begin{proof}
We follow a similar argument as in \cite{bp07}. Consider a cut-off function $\psi_R(x) \in C_0^\infty(\R^d)$ such that $\psi_R(x)=|x|^2$ for $|x|<R$ and $\psi_R(x)=0$ for $|x| \ge 2R$. Then we compute
\begin{align}\label{eq2}
\int \Delta U_s^{m_1} dx=\int U_s^{m_1} \Delta \psi_R(x) dx,\quad \int \Delta V_s^{m_2} dx=\int V_s^{m_2} \Delta \psi_R(x) dx.
\end{align}
We also have
\begin{align*}
\int \psi_R(x) \nabla \cdot (U_s(x) \nabla C_s(x)) dx &=-\int \nabla \psi_R(x) \cdot U_s \nabla C_s dx \\
&=c_d (d-2) \iint U_s(x) \frac{\nabla \psi_R(x) \cdot (x-y)}{|x-y|^d}  V_s(y) dxdy
\end{align*}
and
\begin{align*}
\int \psi_R(x) \nabla \cdot (V_s(x) \nabla Z_s(x)) dx &=-\int \nabla \psi_R(x) \cdot V_s \nabla Z_s dx \\
&=c_d (d-2) \iint U_s(y) \frac{\nabla \psi_R(x) \cdot (x-y)}{|x-y|^d}  V_s(x) dxdy.
\end{align*}
Summing these two expressions yields
\begin{align}\label{eq3}
& \int \psi_R(x) \nabla \cdot (U_s(x) \nabla C_s(x)) dx+\int \psi_R(x) \nabla \cdot (V_s(x) \nabla Z_s(x)) dx \nonumber \\
=& c_d (d-2) \iint \frac{(\nabla \psi_R(x)-\nabla \psi_R(y))\cdot (x-y)}{|x-y|^2} \frac{U_s(x) V_s(y)}{|x-y|^{d-2}} dxdy.
\end{align}
Since $\Delta \psi_R(x)$ and $\frac{(\nabla \psi_R(x)-\nabla \psi_R(y))\cdot (x-y)}{|x-y|^2}$ are bounded, all terms on the right-hand sides of \eqref{eq2} and \eqref{eq3} are bounded. Therefore, as $R \to \infty,$ we may pass to the limit in each term using the Lebesgue dominated convergence theorem to complete the proof. 
\end{proof}

Combining the identity \eqref{eq0} with the variational structure of the free energy, we deduce the following results.

\begin{proposition}[Euler-Lagrange equations for stationary solutions]\label{prop1006}
Let $\Omega_1, \Omega_2 \subset \R^d$ be two connected open sets. Suppose that $U_s \in L_+^1 \cap L^{m_1} (\R^d)$ and $V_s \in L_+^1 \cap L^{m_2}(\R^d)$ with $\int U_s dx=M_1$ and $\int V_s dx=M_2$. Assume also that $U_s>0$ in $\Omega_1$, $U_s=0$ in $\R^d \setminus \Omega_1$, and $V_s>0$ in $\Omega_2$, $V_s=0$ in $\R^d \setminus \Omega_2$. If $(U_s,V_s)$ is a critical point of $F(u,v)$, then it satisfies the Euler-Lagrange conditions
\begin{align}
\left\{
  \begin{array}{ll}
    \frac{m_1}{m_1-1} U_s^{m_1-1}-C_s=\overline{C}_1, & \text{a.e. in } \Omega_1, \\
    \frac{m_2}{m_2-1} V_s^{m_2-1}-Z_s=\overline{C}_2, & \text{a.e. in } \Omega_2,
  \end{array}
\right.
\end{align}
where the constants are given by
\begin{align}
\left\{
  \begin{array}{ll}
    \overline{C}_1=\frac{1}{M_1} \left( \frac{m_1}{m_1-1} \int U_s^{m_1} dx-\int C_s U_s dx \right), \\
    \overline{C}_2=\frac{1}{M_2} \left( \frac{m_2}{m_2-1} \int V_s^{m_2} dx-\int Z_s V_s dx \right).
  \end{array}
\right.
\end{align}
Moreover, if $1/m_1+1/m_2=(d+2)/d$, then $\overline{C}_1=\overline{C}_2=0$ and $\Omega_1=\Omega_2=\R^d$. Consequently,
\begin{align}\label{UsVsm1m2}
\frac{m_1}{m_1-1} \int U_s^{m_1} dx=\frac{m_2}{m_2-1} \int V_s^{m_2} dx.
\end{align}
\end{proposition}
\begin{proof}
We adapt some ideas from \cite{BL13,CCV15}. For any $\psi_1 \in C_0^\infty(\Omega_1)$ and $\psi_2 \in C_0^\infty(\Omega_2)$, define
\begin{align}
\varphi_1(x)=\psi_1(x)-\frac{U_s}{M_1} \int_{\Omega_1} \psi_1 dx, \\
\varphi_2(x)=\psi_2(x)-\frac{V_s}{M_2} \int_{\Omega_2} \psi_2 dx. 
\end{align}
Then $\text{supp}(\varphi_1) \subseteq \Omega_1$, $\text{supp}(\varphi_2) \subseteq \Omega_2$, and $\int_{\Omega_1} \varphi_1 dx=0$, $\int_{\Omega_2} \varphi_2 dx=0$. Moreover, there exist
\begin{align*}
\varepsilon_1: = \frac{ \displaystyle \min_{y \in  \operatorname{supp}(\varphi_1)} U_s(y)} { \displaystyle \max_{y \in \operatorname{supp}(\varphi_1)} \left| \varphi_1 (y) \right|}>0,\quad \varepsilon_2: = \frac{ \displaystyle \min_{y \in  \operatorname{supp}(\varphi_2)} V_s(y)} { \displaystyle \max_{y \in \operatorname{supp}(\varphi_2)} \left| \varphi_2 (y) \right|}>0
\end{align*}
such that $U_s+\varepsilon \varphi_1 \ge 0$ in $\Omega_1$ and $V_s+\varepsilon \varphi_2 \ge 0$ in $\Omega_2$ for all $0<\varepsilon<\min(\varepsilon_1,\varepsilon_2)$. Now, $(U_s,V_s)$ is a critical point of $F(u,v)$ if and only if
\begin{align}
\frac{d}{d \varepsilon} \Big |_{\varepsilon=0} F(U_s+\varepsilon \varphi_1,V_s)=0,\quad \int_{\Omega_1} \varphi_1(x) dx=0, \\
\frac{d}{d \varepsilon} \Big |_{\varepsilon=0} F(U_s,V_s+\varepsilon \varphi_2)=0,\quad \int_{\Omega_2} \varphi_2(x) dx=0.
\end{align}
A direct computation yields
\begin{align*}
\int_{\Omega_1} \left( \mu_{1s}-\frac{1}{M_1} \int_{\Omega_1} \mu_{1s} U_s dx \right) \psi_1(x) dx=0,\quad \text{for any }\psi_1 \in C_0^\infty(\Omega_1), \\
\int_{\Omega_2} \left( \mu_{2s}-\frac{1}{M_2} \int_{\Omega_2} \mu_{2s} V_s dx \right) \psi_2(x) dx=0,\quad \text{for any } \psi_2 \in C_0^\infty(\Omega_2),
\end{align*}
where the chemical potentials are defined as
\begin{align}
\mu_{1s}=\frac{m_1}{m_1-1} U_s^{m_1-1}-C_s, \quad \mu_{2s}=\frac{m_2}{m_2-1} V_s^{m_2-1}-Z_s.
\end{align}
Therefore, we obtain
\begin{align}\label{eq4}
\frac{m_1}{m_1-1} U_s^{m_1-1}-C_s=\frac{1}{M_1} \left( \frac{m_1}{m_1-1} \int U_s^{m_1} dx-\int U_s C_s dx \right)=\overline{C}_1,\quad \text{a.e. in } \Omega_1,
\end{align}
and 
\begin{align}\label{eq5}
\frac{m_2}{m_2-1} V_s^{m_2-1}-Z_s=\frac{1}{M_2} \left( \frac{m_2}{m_2-1} \int V_s^{m_2} dx-\int V_s Z_s dx \right)=\overline{C}_2,\quad \text{a.e. in } \Omega_2.
\end{align}
On the other hand, we can infer from \eqref{eq0} that
\begin{align}
\overline{C}_1=\frac{1}{M_1} \left( \left( \frac{m_1}{m_1-1}-\frac{d}{d-2}\right)\int U_s^{m_1} dx-\frac{d}{d-2}\int V_s^{m_2} dx \right), \\
\overline{C}_2=\frac{1}{M_2} \left( \left( \frac{m_2}{m_2-1}-\frac{d}{d-2}\right)\int V_s^{m_2} dx-\frac{d}{d-2}\int U_s^{m_1} dx \right).
\end{align}
Thanks to the condition $1/m_1+1/m_2=(d+2)/d$, we find 
\begin{align}\label{eq6}
-M_1 \overline{C}_1=\frac{\frac{m_1}{m_1-1}-\frac{d}{d-2}}{\frac{d}{d-2}} M_2 \overline{C}_2=\frac{\frac{d}{d-2}}{\frac{m_2}{m_2-1}-\frac{d}{d-2}} M_2 \overline{C}_2.
\end{align}
Since $U_s \in L^1 \cap L^{m_1}(\R^d),$ the weak Young inequality implies that $\frac{1}{|x|^{d-2}} \ast U_s \in L^r(\R^d)$ for each $r \in \left( \frac{d}{d-2},\frac{1}{\frac{1}{m_1}-\frac{2}{d}} \right].$ Similarly, from $V_s \in L^1 \cap L^{m_2}(\R^d)$ we deduce that $\frac{1}{|x|^{d-2}} \ast V_s \in L^q(\R^d)$ for each $q \in \left( \frac{d}{d-2},\frac{1}{\frac{1}{m_2}-\frac{2}{d}} \right].$ In particular, both $\frac{1}{|x|^{d-2}} \ast V_s$ and $U_s^{m_1-1}$ belong to $L^{m_1/(m_1-1)}(\R^d).$ Analogously, both $\frac{1}{|x|^{d-2}} \ast U_s$ and $V_s^{m_2-1}$ belong to $L^{m_2/(m_2-1)}(\R^d).$ These properties with \eqref{eq4} and \eqref{eq5} allow us to assert that $\Omega_1=\Omega_2=\R^d.$ If not, we choose a sequence of points $x_{1n} \to x_1 \in \partial \Omega_1$ with $x_{1n} \in \Omega_1$. Then $\frac{m_1}{m_1-1} U_s^{m_1-1}(x_{1n}) \to 0$, whereas $\int \frac{V_s(y)}{|x_{1n}-y|^{d-2}}dy>0$. It then follows from \eqref{eq4} that $\overline{C}_1 < 0$. Similarly, we choose a sequence of points $x_{2n} \to x_2 \in \partial \Omega_2$ with $x_{2n} \in \Omega_2$, we deduce from \eqref{eq5} that $\overline{C}_2 < 0$. Hence, \eqref{eq6} together with the facts that $1/m_1>2/d$ and $1/m_2>2/d$ give rise to $\overline{C}_1=\overline{C}_2=0$. Therefore,  
\begin{align}
\left\{
  \begin{array}{ll}
    \frac{m_1}{m_1-1} U_s^{m_1-1}-C_s=0, &\text{in } \R^d, \\
   \frac{m_2}{m_2-1} V_s^{m_2-1}-Z_s=0, &\text{in } \R^d.
  \end{array}
\right.
\end{align}
This completes the proof. 
\end{proof}

By Proposition \ref{prop1006}, one has 
\begin{align}\label{eq7}
\left\{
  \begin{array}{ll}
    -\Delta C_s=\left( \frac{m_2-1}{m_2} Z_s  \right)^{\frac{1}{m_2-1}}, & \text{in } \R^d, \\
    -\Delta Z_s=\left( \frac{m_1-1}{m_1} C_s  \right)^{\frac{1}{m_1-1}}, & \text{in } \R^d.
  \end{array}
\right.
\end{align}
It was proved by Chen and Li \cite{CLPDE05} that all positive solutions $C_s, Z_s$ of \eqref{eq7} are radially symmetric and monotone decreasing about some point. In particular, if $m_1=m_2=2d/(d+2)$, then $V_s=C_s$ and they both assume the standard form \cite{CLAA05}
\begin{align}
c \left( \frac{\lambda}{\lambda^2+|x-x_0|^2}  \right)^{\frac{d-2}{2}}
\end{align} 
with some constant $c=c(d)$, and for some $\lambda>0$ and $x_0 \in \R^d$. This completes our analysis of the stationary solutions and the proof of Theorem \ref{steadymain}.

\subsection{Extremal functions for the Hardy-Littlewood-Sobolev inequality}\label{sec32}

This subsection establishes the connection between the extremal functions of the Hardy-Littlewood-Sobolev (HLS) inequality \eqref{fh} and the steady states of system \eqref{eq7}. We first introduce the functionals
\begin{align}\label{hu}
h(u,v):=\iint \frac{u(x)v(y)}{|x-y|^{d-2}} dxdy
\end{align}
and
\begin{align}\label{Ju}
J(u,v):=\frac{h(u,v)}{\|u\|_{m_1} \|v\|_{m_2}}.
\end{align}
The following result of Lieb \cite{lieb83} on the existence and properties of the extremal functions for the HLS inequality is fundamental to our analysis.

\begin{proposition}[Existence of the optimizer] \label{Cstar} \quad Let $1/m_1+1/m_2=(d+2)/d$ and suppose $u \in L^{m_1}(\R^d), v \in L^{m_2}(\R^d)$. Then there exists a sharp constant $C_*$ such that
\begin{align}\label{VHLS}
J(U,V) = C_*,
\end{align}
where $(U,V)$ are non-negative, radially symmetric (up to translation) and non-increasing functions. Moreover, the sharp constant admits the upper bound
\begin{align}
C_* \le \frac{d}{2}\left(\frac{2 \pi^{d/2}}{d \Gamma(d/2)}\right)^{(d-2)/d} \frac{1}{m_1 m_2}\left( \left(\frac{(d-2)/d}{1-1/q}\right)^{(d-2)/d}+\left( \frac{(d-2)/d}{1-1/r} \right)^{(d-2)/d} \right).
\end{align}
\end{proposition}

We now establish the variational structure of the HLS inequality. Define the critical threshold $(x_*,y_*)$ for $(\|u\|_{m_1}^{m_1},\|v\|_{m_2}^{m_2})$ by
\begin{align}\label{xstarstar}
x_*:=\left( \frac{m_1}{c_d C_* (m_1-1)} A^{-\frac{1}{m_2}}  \right)^{ \frac{1}{\frac{1}{m_1}+\frac{1}{m_2}-1} },\quad y_*=Ax_*, \quad A=\frac{m_1(m_2-1)}{m_2(m_1-1)},
\end{align}
and define the set
\begin{align*}
\mathcal{Y}_{(x_*,y_*)}:=\left\{ u \in L^{m_1}(\R^d),v\in L^{m_2}(\R^d)~\Big|~\|u\|_1=M_1, \|v\|_1=M_2,~ \|u\|_{m_1}^{m_1}=x_*, \|v\|_{m_2}^{m_2}=y_* \right\}.
\end{align*}
\begin{proposition}[Identification of the optimizer] \label{ELE}
Let $(U,V)$ be an optimizer of $J(u,v)$ in $\mathcal{Y}_{(x_*,y_*)}$. Then $(U,V)$ is a pair of radially symmetric, non-increasing stationary solutions to \eqref{uvsystem} and satisfies the Euler-Lagrange equations
\begin{align} 
\left\{
  \begin{array}{ll}
    \frac{m_1}{m_1-1} U^{m_1-1}=c_d \int \frac{V(y)}{|x-y|^{d-2}}dy, & x \in \R^d, \\
  \frac{m_2}{m_2-1} V^{m_2-1}=c_d \int \frac{U(y)}{|x-y|^{d-2}}dy, & x \in \R^d,
  \end{array}
\right.
\end{align}
\end{proposition}
\begin{proof}
We now derive the Euler-lagrange equations for $(U,V)$. We introduce
$$\Omega_1:=\left\{ x\in \R^d ~|~ U(x)>0 \right\},\quad \Omega_2:=\left\{ x\in \R^d ~|~ V(x)>0 \right\}. $$
For any $\psi_1 \in C_0^\infty(\Omega_1),$ define
\begin{align}\label{upsi}
\varphi_1(x)=\psi_1(x)-\frac{U}{M_1} \int_{\Omega_1} \psi_1 dx.
\end{align}
Then $\text{supp}(\varphi_1) \subseteq \Omega_1$ and $\int_{\Omega_1} \varphi_1 dx=0$. Besides, there exists
\begin{align*}
\varepsilon_0: = \frac{ \displaystyle \min_{y \in  \operatorname{supp}(\varphi_1)} U(y)} { \displaystyle \max_{y \in \operatorname{supp}(\varphi_1)} \left| \varphi_1 (y) \right| }  >0
\end{align*}
such that $U+\varepsilon \varphi_1 \ge 0$ in $\Omega_1$ for $0<\varepsilon<\varepsilon_0$. Computing the first variation yields
\begin{align}\label{functional}
0=&\frac{d}{d\varepsilon} \Bigg|_{\varepsilon=0} J(U+\varepsilon \varphi_1,V) \nonumber \\
=& \int \left( \|U\|_{m_1} \|V\|_{m_2} \int \frac{V(y)}{|x-y|^{d-2}} dy-h(U,V) \|V\|_{m_2} \|U\|_{m_1}^{1-m_1} U^{m_1-1} \right)\varphi_1(x) dx,
\end{align}
By the definition of $\varphi_1$ in \eqref{upsi}, equation \eqref{functional} implies that $(U,V)$ solves the following equation
\begin{align}\label{VV1}
\int \frac{V(y)}{|x-y|^{d-2}} dy & =C_* \|U\|_{m_1}^{1-m_1} \|V\|_{m_2} U^{m_1-1} \nonumber \\
&=\frac{1}{c_d} \frac{m_1}{m_1-1} U^{m_1-1},\quad \text{a.e. } x \in \R^d,
\end{align}
since $(U,V) \in \mathcal{Y}_{(x_*,y_*)}$. Similarly, by an analogous argument as in \eqref{upsi}-\eqref{VV1}, we also have
\begin{align}\label{WW240704}
 \int \frac{U(y)}{|x-y|^{d-2}} dy =\frac{1}{c_d} \frac{m_2}{m_2-1} V^{m_2-1},\quad \text{a.e. } x \in \R^d.
\end{align}
Recalling \eqref{eq7}, we conclude that $(U,V)$ is the steady state of \eqref{uvsystem}. 
\end{proof}

\section{Global existence and finite-time blow-up: proof of Theorem \ref{main}} \label{sec4}

In this section, we use the characterization of steady states and the free energy to distinguish between two cases: 
\begin{itemize}
\item $\|u_0\|_{m_1}>\|U_s\|_{m_1}$ and $\|v_0\|_{m_2}>\|V_s\|_{m_2}$ (finite-time blow-up),
\item $\|u_0\|_{m_1}<\|U_s\|_{m_1}$ and $\|v_0\|_{m_2}<\|V_s\|_{m_2}$ (global existence).
\end{itemize}  
These are treated in Sections \ref{sub41} and  \ref{sub42}, respectively. In order to prove the main results (Proposition \ref{blowup} and Proposition \ref{globalexistence}), we first establish some a priori estimates.
\begin{proposition}\label{belowabove}
Let $1/m_1+1/m_2=(d+2)/d$. Assume $(u,v)$ is a solution to \eqref{uvsystem}. Under assumption \eqref{initialdata} and
\begin{align}\label{Fu0}
 F(u_0,v_0)<F(U_s,V_s),
\end{align}
the following results hold:
\begin{enumerate}
  \item[\textbf{(i)}] If
  \begin{align}\label{xiaoyum1}
   \|u_0\|_{m_1}<\|U_s\|_{m_1} \text{  and  } \|v_0\|_{m_2}<\|V_s\|_{m_2},
  \end{align}
  then there exists a constant $0<\mu_1<1$ such that the corresponding solution $(u,v)$ satisfies that for any $t>0,$
     \begin{align}\label{xiaoyum}
      \|u(\cdot,t)\|_{m_1}< \mu_1  \|U_s\|_{m_1} \text{  and  } \|v(\cdot,t)\|_{m_2}< \mu_1 \|V_s\|_{m_2}.
    \end{align}
  \item[\textbf{(ii)}] If
  \begin{align}\label{dayum1}
  \|u_0\|_{m_1}>\|U_s\|_{m_1} \text{  and  } \|v_0\|_{m_2}>\|V_s\|_{m_2},
  \end{align}
  then there exists a constant $\mu_2>1$ such that the corresponding solution $(u,v)$ satisfies that for any $t>0,$
\begin{align}\label{dayum}
  \|u(\cdot,t)\|_{m_1}> \mu_2  \|U_s\|_{m_1} \text{  and  } \|v(\cdot,t)\|_{m_2}> \mu_2 \|V_s\|_{m_2}.
\end{align}
\end{enumerate}
\end{proposition}
\begin{proof}
First, by setting
\begin{align*}
f=u(x), h=v(y), \beta=d-2, q=m_1, r=m_2
\end{align*}
in the HLS inequality \eqref{fh} and applying Proposition \ref{Cstar}, we infer from the expression of $F(u,v)$ that
\begin{align}\label{111}
F(u,v) &=\frac{1}{m_1-1} \int u^{m_1} dx+ \frac{1}{m_2-1} \int v^{m_2}dx-c_d \iint \frac{u(x)v(y)}{|x-y|^{d-2}}dxdy \nonumber\\
& \ge \frac{1}{m_1-1} \int u^{m_1} dx+ \frac{1}{m_2-1} \int v^{m_2}dx-c_d C_* \|u\|_{m_1}\|v\|_{m_2},
\end{align}
On the other hand, applying \eqref{UsVsm1m2}, it follows from Proposition \ref{ELE} that
\begin{align}\label{222}
F(U_s,V_s) &=\frac{1}{m_1-1} \int U_s^{m_1} dx+ \frac{1}{m_2-1} \int V_s^{m_2}dx-c_d C_* \|U_s\|_{m_1}\|V_s\|_{m_2} \nonumber \\
&=\frac{m_1+m_2}{m_2(m_1-1)} \int U_s^{m_1} dx-c_d C_* \left( \frac{m_1(m_2-1)}{m_2 (m_1-1)} \right)^{\frac{1}{m_2}} \|U_s\|_{m_1}^{\frac{m_1+m_2}{m_2}}.
\end{align}
Define the auxiliary function
\[
g(x,y)=\frac{1}{m_1-1} x+\frac{1}{m_2-1}y-c_d C_* x^{\frac{1}{m_1}} y^{\frac{1}{m_2}}.
\]
Since $F(u_0,v_0)<F(U_s,V_s),$ there exists $0<\delta<1$ such that
\begin{align}\label{delta}
F(u_0,v_0)<\delta F(U_s,V_s).
\end{align}
By the monotonicity of $F(u,v)$ in time, we have
\begin{align}\label{444}
g\left( \|u\|_{m_1}^{m_1},\|v\|_{m_2}^{m_2} \right) \le F(u,v) \le F(u_0,v_0)<\delta F(U_s,V_s)=\delta g\left( \|U_s\|_{m_1}^{m_1},\|V_s\|_{m_2}^{m_2} \right).
\end{align}
Therefore, in the region where $x > \|U_s\|_{m_1}^{m_1}$ and $y > \|V_s\|_{m_2}^{m_2}$, the function $g(x, y)$ is strictly decreasing in each variable. Hence, the inequality $g(x, y) < \delta g(x_*, y_*)$ with $0 < \delta < 1$ implies that $x > \mu_2 x_*$ and $y > \mu_2 y_*$ for some $\mu_2 > 1$ depending on $\delta$. Consequently, if $\|u_0\|_{m_1} > \|U_s\|_{m_1}$ and $\|v_0\|_{m_2} > \|V_s\|_{m_2}$, then
\begin{align}\label{555}
\|u(\cdot,t)\|_{m_1} > \mu_2  \|U_s\|_{m_1} \quad \text{and} \quad \|v(\cdot,t)\|_{m_2} > \mu_2 \|V_s\|_{m_2}
\end{align}
for all $t>0$. Conversely, in the region where $x < \|U_s\|_{m_1}^{m_1}$ and $y < \|V_s\|_{m_2}^{m_2}$, the function $g(x, y)$ is strictly increasing in each variable. Therefore, the inequality $g(x, y) < \delta g(x_*, y_*)$ implies that $x < \mu_1 x_*$ and $y < \mu_1 y_*$ for some $0<\mu_1 < 1$ depending on $\delta$. Thus, if $\|u_0\|_{m_1} < \|U_s\|_{m_1}$ and $\|v_0\|_{m_2} < \|V_s\|_{m_2}$, then
\begin{align}\label{666}
\|u(\cdot,t)\|_{m_1} < \mu_1  \|U_s\|_{m_1} \quad \text{and} \quad \|v(\cdot,t)\|_{m_2} < \mu_1 \|V_s\|_{m_2}
\end{align}
for all $t > 0$. This completes the proof. 
\end{proof}

\subsection{Proof of $\textbf{(\emph{ii})}$ of Theorem \ref{main} (finite-time blow-up)} \label{sub41}

In this subsection, we prove the finite-time blow-up result in Theorem \ref{main} for large initial data. We employ a standard argument based on the evolution of the second moment of solutions, following the approach in \cite{JL92}. Define the second moment as
\begin{align}
m_2(t):=\int \left(|x|^2 u(x,t)+|x|^2 v(x,t) \right) dx.
\end{align}
\begin{proposition}[Finite time blow-up]\label{blowup}
Let $1<m_1<2-2/d$, $1<m_2<2-2/d$ and $1/m_1+1/m_2=(d+2)/d$. Assume that $(u,v)$ is a weak solution obtained in Lemma \ref{ueps} on $[0,T_w)$. Under assumptions \eqref{initialdata}, \eqref{Fu0} and \eqref{dayum1}, we have $T_w<\infty$ and the solution to \eqref{uvsystem} blows up in finite time in the sense that 
\begin{align*}
\displaystyle \limsup_{t \to T_w} \left(\|u(\cdot,t)\|_{L^\infty(\R^d)}+\|v(\cdot,t)\|_{L^\infty(\R^d)} \right) =\infty.
\end{align*}
\end{proposition}
\begin{proof}
We present formal computations here, the passage to the limit from the approximated problem \eqref{uveps} can be carried out similarly as in \cite{su06,suku06}. The key tool is the time evolution of the second moment. Integrating by parts in \eqref{uvsystem} and leveraging mass conservation, we compute
\begin{align*}
\frac{d}{dt} m_2(t) &=2d \int u^{m_1} dx+2d \int v^{m_2} dx+2 c_d (2-d) \iint u(x)v(y) \frac{x \cdot (x-y)}{|x-y|^d} dxdy \\
&~~+2 c_d (2-d) \iint u(y)v(x) \frac{x \cdot (x-y)}{|x-y|^d} dxdy \\
&=2d \int u^{m_1} dx+2d \int v^{m_2} dx-2(d-2)c_d \iint \frac{u(x)u(y)}{|x-y|^{d-2}} dxdy \\
&=\left( 2d-\frac{2(d-2)}{m_1-1} \right) \int u^{m_1} dx+\left(   2d-\frac{2(d-2)}{m_2-1}\right) \int v^{m_2} dx+2(d-2) F(u,v).
\end{align*}
By Proposition \ref{belowabove}, assumptions \eqref{Fu0} and \eqref{dayum1} imply that \eqref{dayum} holds for some $\mu_2>1$. We use the fact that $F(u,v)$ is non-increasing in time to obtain
\begin{align}
\frac{d}{dt} m_2(t) & \le \left( 2d-\frac{2(d-2)}{m_1-1} \right) \mu_2 \int U_s^{m_1} dx+\left( 2d-\frac{2(d-2)}{m_2-1}\right) \mu_2 \int V_s^{m_2} dx+2(d-2) F(u_0,v_0) \nonumber \\
&< \left( 2d-\frac{2(d-2)}{m_1-1} \right) \mu_2 \int U_s^{m_1} dx+\left( 2d-\frac{2(d-2)}{m_2-1}\right) \mu_2 \int V_s^{m_2} dx+2(d-2) F(U_s,V_s) \nonumber \\
&=\mu_2 \left( 2d-\frac{2(d-2)}{m_1-1}+\left( 2d-\frac{2(d-2)}{m_2-1} \right)\frac{m_1(m_2-1)}{m_2(m_1-1)}    \right) \int U_s^{m_1} dx \nonumber \\
&~~~+2(d-2) \left( \frac{m_1+m_2-m_1m_2}{(m_1-1)m_2} \right) \int U_s^{m_1} dx \nonumber \\
&=-4\frac{(d-2)m_1}{d(m_1-1)} (\mu_2-1)<0,
\end{align}
where we have used identities \eqref{UsVs} and \eqref{FUsVs}, together with the conditions $1<m_1<2-2/d$, $1<m_2<2-2/d$ and $1/m_1+1/m_2=(d+2)/d$. Consequently, there exists a finite time $T>0$ such that $\displaystyle \lim_{t \to T} m_2(t)=0$.

To connect the vanishing of the second moment to $L^\infty$-blow-up, we apply H\"{o}lder's inequality to get
\begin{align}
\int (u+v) dx &=\int_{|x| \le R} (u+v) dx+\int_{|x|> R} (u+v) dx \nonumber \\
& \le C(d) \left(\|u\|_{L^\infty(\R^d)}+\|v\|_{L^\infty(\R^d)}\right) R^d+\frac{1}{R^2} \int (|x|^2 u+|x|^2 v)) dx.
\end{align}
Choosing $R=\left( \frac{C m_2(t)}{\|u\|_{L^\infty(\R^d)}+\|v\|_{L^\infty(\R^d)}}\right)^{1/(d+2)}$ results in 
\begin{align}
M_1+M_2 \le C m_2(t)^{\frac{d}{d+2}} \left(\|u\|_{L^\infty(\R^d)}+\|v\|_{L^\infty(\R^d)}\right)^{\frac{2}{d+2}},
\end{align}
where the constant $C$ depends on $d$ and the masses. Consequently, 
\begin{align}
\limsup\limits_{t \to T} (\|u\|_{L^\infty(\R^d)}+v\|_{L^\infty(\R^d)}) \ge \limsup\limits_{t \to T} \frac{ (M_1+M_2)^{(d+2)/2}}{C m_2(t)^{d/2}}=\infty.
\end{align}
This completes the proof.
\end{proof}

\subsection{Proof of $\textbf{\emph{(i)}}$ of Theorem \ref{main} (global existence)} \label{sub42}

A direct consequence of the uniform bounds in \eqref{xiaoyum} and the local existence theory in Lemma \ref{ueps} is the following global existence result. 

\begin{proposition}[Global existence]\label{globalexistence}
Let $1<m_1<2-2/d$, $1<m_2<2-2/d$ and $1/m_1+1/m_2=(d+2)/d$. Assume $(u,v)$ is the weak solution obtained in Lemma \ref{ueps} on $[0,T_w)$. Under assumptions \eqref{initialdata}, \eqref{Fu0} and \eqref{xiaoyum1}, this weak solution exists globally in time, i.e., $T_w = \infty$. 
\end{proposition}
\begin{proof}
First, Proposition \ref{belowabove} together with assumptions \eqref{Fu0} and \eqref{xiaoyum1} implies the uniform bounds in \eqref{xiaoyum}. That is, $\|u(\cdot,t)\|_{m_1}$ and $\|v(\cdot,t)\|_{m_2}$ are uniformly bounded for all $t>0$. We will show that this implies uniform boundedness of $\|u\|_p$ and $\|v\|_q$ for all $p,q>1$. Then, using the iteration method, we prove that the solution $(u,v)$ is uniformly bounded in time, which ensures that the local weak solution $(u,v)$ from Lemma \ref{ueps} can be extended globally to $[0, \infty)$. 

{\it\textbf{Step 1}} ($L^p$ and $L^q$ estimates for $p>1$, $q>1$) \quad 
We derive a priori estimates for solutions to \eqref{uvsystem}, following a similar regularity argument as in \cite{JK22}. Multiplying \eqref{uvsystem}$_1$ by $p u^{p-1}$ and \eqref{uvsystem}$_2$ by $q v^{q-1}$, we obtain
\begin{align}
 \frac{d}{dt} \int u^p dx+\frac{4m_1 p(p-1)}{(p+m_1-1)^2} \int \left| \nabla u^{\frac{p+m_1-1}{2}} \right|^2 dx 
=  (p-1) \int u^p v dx,
\end{align}
and 
\begin{align}
 \frac{d}{dt} \int v^q dx+\frac{4m_2 q(q-1)}{(q+m_2-1)^2} \int \left| \nabla v^{\frac{q+m_2-1}{2}} \right|^2 dx 
=  (q-1) \int v^q u dx.
\end{align}
Summing these two equations yields
\begin{align}\label{upvq}
 &\frac{d}{dt} \int u^p dx+\frac{d}{dt} \int v^q dx+\frac{4m_1 p(p-1)}{(p+m_1-1)^2} \int \left| \nabla u^{\frac{p+m_1-1}{2}} \right|^2 dx 
 +\frac{4m_2 q(q-1)}{(q+m_2-1)^2} \int \left| \nabla v^{\frac{q+m_2-1}{2}} \right|^2 dx \nonumber \\
&= (p-1) \int u^p v dx+(q-1) \int v^q u dx.
\end{align}

We now estimate $\int u^p vdx$ and $\int v^q u dx$. By H\"{o}lder's inequality, one has
\begin{align}
\int u^p v dx \le \|u\|_{p k_1}^p \|v\|_{k_1'}
\end{align}
for $1/k_1+1/k_1'=1.$ Applying the Sobolev inequality and H\"{o}lder's inequality with 
\begin{align}\label{prange}
\overline{m}_1<p k_1<\frac{(p+m_1-1)d}{d-2},\quad k_1<\frac{d}{d-2},
\end{align}
we discover
\begin{align}\label{up}
\|u\|_{p k_1}^p & \le \|u\|_{\overline{m}_1}^{p(1-\theta_1)} \|u\|_{\frac{(p+m_1-1)d}{d-2}}^{p \theta_1} \nonumber\\
& = \|u\|_{\overline{m}_1}^{p(1-\theta_1)} \left \|u^{\frac{p+m_1-1}{2}} \right\|_{\frac{2d}{d-2}}^{\frac{2 p \theta_1}{p+m_1-1}} \nonumber \\
& \le C(d) \|u\|_{\overline{m}_1}^{p(1-\theta_1)} \left \| \nabla u^{\frac{p+m_1-1}{2}} \right\|_{2}^{\frac{2 p \theta_1}{p+m_1-1}},
\end{align}
where the parameters satisfy 
\begin{align*}
\theta_1=\frac{\frac{1}{\overline{m}_1}-\frac{1}{pk_1}}{\frac{1}{\overline{m}_1}-\frac{d-2}{(p+m_1-1)d}} \in (0,1), \quad
\frac{2 p \theta_1}{p+m_1-1}=\frac{\frac{p}{\overline{m}_1}-\frac{1}{k_1}}{\frac{p+m_1-1}{2 \overline{m}_1}-\frac{d-2}{2d}}.
\end{align*}
Moreover, for $\overline{m}_2<\frac{d}{2}<k_1'<\frac{(q+m_2-1)d}{d-2},$ H\"{o}lder's inequality yields 
\begin{align}\label{v}
\|v\|_{k_1'} & \le \|v\|_{\overline{m}_2}^{1-\theta_2} \|v\|_{\frac{(q+m_2-1)d}{d-2}}^{\theta_2} \nonumber \\
& =\|v\|_{\overline{m}_2}^{1-\theta_2} \|v^{\frac{q+m_2-1}{2}}\|_{\frac{2d}{d-2}}^{\frac{2 \theta_2}{q+m_2-1}} \nonumber \\
& \le C(d) \|v\|_{\overline{m}_2}^{1-\theta_2} \left\|\nabla v^{\frac{q+m_2-1}{2}} \right\|_{2}^{\frac{2 \theta_2}{q+m_2-1}}, 
\end{align}
where the parameters are given by
\begin{align*}
\theta_2=\frac{\frac{1}{\overline{m}_2}-\frac{1}{k_1'}}{\frac{1}{\overline{m}_2}-\frac{d-2}{(q+m_2-1)d}} \in (0,1), \quad 
\frac{2 \theta_2}{q+m_2-1}=\frac{\frac{1}{\overline{m}_2}-\frac{1}{k_1'}}{\frac{q+m_2-1}{2 \overline{m}_2}-\frac{d-2}{d}}.
\end{align*}
Collecting \eqref{up} and \eqref{v} together, we arrive at
\begin{align}\label{upv}
\int u^p v dx \le C(d) \|u\|_{\overline{m}_1}^{p(1-\theta_1)} \left \| \nabla u^{\frac{p+m_1-1}{2}} \right\|_{2}^{\frac{2 p \theta_1}{p+m_1-1}} \|v\|_{\overline{m}_2}^{1-\theta_2} \left\|\nabla v^{\frac{q+m_2-1}{2}} \right\|_{2}^{\frac{2 \theta_2}{q+m_2-1}},
\end{align}
with the exponents satisfying
\begin{align}
\frac{2 p \theta_1}{p+m_1-1}=\frac{\frac{p}{\overline{m}_1}-\frac{1}{k_1}}{\frac{p+m_1-1}{2 \overline{m}_1}-\frac{d-2}{2d}}, \\
\frac{2 \theta_2}{q+m_2-1}=\frac{\frac{1}{\overline{m}_2}-\frac{1}{k_1'}}{\frac{q+m_2-1}{2 \overline{m}_2}-\frac{d-2}{d}}.
\end{align}
Similarly, for parameters satisfying
\begin{gather}
\frac{1}{k_2}+\frac{1}{k_2'}=1, \nonumber \\
\overline{m}_2<q k_2<\frac{(q+m_2-1)d}{d-2}, \label{qrange}\\
\overline{m}_1<\frac{d}{2}<k_2'<\frac{(p+m_1-1)d}{d-2}, \nonumber
\end{gather}
we apply the Sobolev inequality and  H\"{o}lder's inequality to obtain
\begin{align}\label{vqu}
\int v^q u dx & \le \|v^q\|_{q k_2}^q \|u\|_{k_2'} \nonumber \\
&\le \|v\|_{\overline{m}_2}^{q (1-\alpha)} \|v\|_{\frac{(q+m_2-1)d}{d-2}}^{q \alpha} \|u\|_{\overline{m}_1}^{1-\alpha_0} \|u\|_{\frac{(p+m_1-1)d}{d-2}}^{\alpha_0} \nonumber \\
&\le C(d) \|v\|_{\overline{m}_2}^{q (1-\alpha)} \left\|\nabla v^{\frac{q+m_2-1}{2}} \right\|_{2}^{\frac{2q \alpha}{q+m_2-1}}\|u\|_{\overline{m}_1}^{1-\alpha_0} \left\|\nabla u^{\frac{p+m_1-1}{2}}\right\|_{2}^{\frac{2\alpha_0}{p+m_1-1}},     
\end{align}
where 
\begin{align*}
\alpha=\frac{\frac{1}{\overline{m}_2}-\frac{1}{q k_2}}{\frac{1}{\overline{m}_2}-\frac{d-2}{(q+m_2-1)d}},\quad \alpha_0=\frac{\frac{1}{\overline{m}_1}-\frac{1}{ k_2'}}{\frac{1}{\overline{m}_1}-\frac{d-2}{(p+m_1-1)d}}, 
\end{align*} 
and 
\begin{align*}
\frac{2q \alpha}{q+m_2-1}=\frac{\frac{q}{\overline{m}_2}-\frac{1}{k_2}}{\frac{q+m_2-1}{2 \overline{m}_2}-\frac{d-2}{2d}}, \quad \frac{2 \alpha_0}{p+m_1-1}=\frac{\frac{1}{\overline{m}_1}-\frac{1}{k_2'}}{\frac{p+m_1-1}{2 \overline{m}_1}-\frac{d-2}{2d}}.
\end{align*}
Substituting \eqref{upv} and \eqref{vqu} into \eqref{upvq} results in 
\begin{align}\label{251029}
&\frac{d}{dt} \int u^p dx+\frac{d}{dt} \int v^q dx+\frac{4m_1 p(p-1)}{(p+m_1-1)^2} \int \left| \nabla u^{\frac{p+m_1-1}{2}} \right|^2 dx+\frac{4m_2 q(q-1)}{(q+m_2-1)^2} \int \left| \nabla v^{\frac{q+m_2-1}{2}} \right|^2 dx \nonumber \\
 & \le C\left(\|u\|_{\overline{m}_1}, \|v\|_{\overline{m}_2} \right) 
  \left\| \nabla u^{\frac{p+m_1-1}{2}} \right\|_{2}^{ 
    \frac{\frac{p}{\overline{m}_1} - \frac{1}{k_1}}{\frac{p+m_1-1}{2\overline{m}_1} - \frac{d-2}{2d}}
  } 
  \left\| \nabla v^{\frac{q+m_2-1}{2}} \right\|_{2}^{ 
    \frac{\frac{1}{\overline{m}_2} - \frac{1}{k_1'}}{\frac{q+m_2-1}{2\overline{m}_2} - \frac{d-2}{d}}
  }  \nonumber \\
& \quad + 
 C\left(\|u\|_{\overline{m}_1}, \|v\|_{\overline{m}_2} \right) 
\left\| \nabla v^{\frac{q+m_2-1}{2}} \right\|_{2}^{ 
    \frac{\frac{q}{\overline{m}_2} - \frac{1}{k_2}}{\frac{q+m_2-1}{2\overline{m}_2} - \frac{d-2}{2d}}
  } 
  \left\| \nabla u^{\frac{p+m_1-1}{2}} \right\|_{2}^{ 
    \frac{\frac{1}{\overline{m}_1} - \frac{1}{k_2'}}{\frac{p+m_1-1}{2\overline{m}_1} - \frac{d-2}{2d}}
  }.  
\end{align}
We now further estimate \eqref{251029}. Since $1/m_1+1/m_2=(d+2)/d$, we impose the relation
\begin{align}
\frac{q+m_2-1}{\overline{m}_2}=\frac{p+m_1-1}{2 \overline{m}_1}.
\end{align}
Then the inequalities
\begin{align}\label{pm1}
\frac{\frac{p}{\overline{m}_1}-\frac{1}{k_1}}{\frac{p+m_1-1}{2 \overline{m}_1}-\frac{d-2}{2d}}+\frac{\frac{1}{\overline{m}_2}-\frac{1}{k_1'}}{\frac{q+m_2-1}{2 \overline{m}_2}-\frac{d-2}{d}}<2 
\end{align}
and 
\begin{align}\label{qm2}
\frac{\frac{q}{\overline{m}_2}-\frac{1}{k_2}}{\frac{q+m_2-1}{2 \overline{m}_2}-\frac{d-2}{2d}}+\frac{\frac{1}{\overline{m}_1}-\frac{1}{k_2'}}{\frac{p+m_1-1}{2 \overline{m}_1}-\frac{d-2}{2d}}<2
\end{align}
hold if and only if
\begin{align}
\frac{1}{\overline{m}_1}+\frac{1}{\overline{m}_2}<\frac{m_1}{\overline{m}_1}+\frac{2}{d}, \\
\frac{1}{\overline{m}_1}+\frac{1}{\overline{m}_2}<\frac{m_2}{\overline{m}_2}+\frac{2}{d},
\end{align} 
which are guaranteed by choosing 
\begin{align}
\overline{m}_1<m_1 \text{  and  } \overline{m}_2<m_2.
\end{align}

Armed with the uniform bounds in \eqref{xiaoyum} and H\"{o}lder's inequality, we deduce that
\begin{align*}
\|u\|_{\overline{m}_1}\le \|u\|_1^{1-\theta} \|u\|_{m_1}^\theta \le C, \quad 1-\theta+\frac{\theta}{m_1}=\frac{1}{\overline{m}_1},\\
\|v\|_{\overline{m}_2}\le \|v\|_1^{1-\eta} \|v\|_{m_2}^\eta \le C, \quad 1-\eta+\frac{\eta}{m_2}=\frac{1}{\overline{m}_2},
\end{align*}
where the constants $C$ depend only on $\|U_s\|_{m_1}$, $\|V_s\|_{m_2}$, $\|u_0\|_1$ and $\|v_0\|_1$. Given \eqref{pm1}, we can infer from \eqref{251029} that
\begin{align}\label{upv1}
&C\left(\|u\|_{\overline{m}_1}, \|v\|_{\overline{m}_2} \right) 
  \left\| \nabla u^{\frac{p+m_1-1}{2}} \right\|_{2}^{ 
    \frac{\frac{p}{\overline{m}_1} - \frac{1}{k_1}}{\frac{p+m_1-1}{2\overline{m}_1} - \frac{d-2}{2d}}
  } 
  \left\| \nabla v^{\frac{q+m_2-1}{2}} \right\|_{2}^{ 
    \frac{\frac{1}{\overline{m}_2} - \frac{1}{k_1'}}{\frac{q+m_2-1}{2\overline{m}_2} - \frac{d-2}{d}}
  } \nonumber \\
\le &\frac{m_1 p (p-1)}{(p+m_1-1)^2} \left \| \nabla u^{\frac{p+m_1-1}{2}} \right\|_{2}^{2} + \frac{m_2 q (q-1)}{(q+m_2-1)^2} \left\|\nabla v^{\frac{q+m_2-1}{2}} \right\|_{2}^{2}+C\left(\|u\|_{m_1},\|v\|_{m_2}\right),
\end{align}
where we have used the inequality $x^a y^b \le C_1 x^2+C_2 y^2+C_{3}$ for $a+b<2.$ Similarly, \eqref{qm2} ensures that
\begin{align}\label{vqu1}
& C\left(\|u\|_{\overline{m}_1}, \|v\|_{\overline{m}_2} \right) 
\left\| \nabla v^{\frac{q+m_2-1}{2}} \right\|_{2}^{ 
    \frac{\frac{q}{\overline{m}_2} - \frac{1}{k_2}}{\frac{q+m_2-1}{2\overline{m}_2} - \frac{d-2}{2d}}
  } 
  \left\| \nabla u^{\frac{p+m_1-1}{2}} \right\|_{2}^{ 
    \frac{\frac{1}{\overline{m}_1} - \frac{1}{k_2'}}{\frac{p+m_1-1}{2\overline{m}_1} - \frac{d-2}{2d}}
  } \nonumber \\
\le & \frac{m_1 p (p-1)}{(p+m_1-1)^2} \left \| \nabla u^{\frac{p+m_1-1}{2}} \right\|_{2}^{2} + \frac{m_2 q (q-1)}{(q+m_2-1)^2} \left\|\nabla v^{\frac{q+m_2-1}{2}} \right\|_{2}^{2}+C\left(\|u\|_{m_1},\|v\|_{m_2}\right),
\end{align}
Plugging \eqref{upv1} and \eqref{vqu1} into \eqref{251029}, we have
\begin{align}
&\frac{d}{dt} \int u^p dx+\frac{d}{dt} \int v^q dx+\frac{2m_1 p(p-1)}{(p+m_1-1)^2} \int \left| \nabla u^{\frac{p+m_1-1}{2}} \right|^2 dx \nonumber \\
&+\frac{2m_2 q(q-1)}{(q+m_2-1)^2} \int \left| \nabla v^{\frac{q+m_2-1}{2}} \right|^2 dx \le C. \label{pqtime}
\end{align}
On the other hand, from \eqref{prange}, the Sobolev inequality and Young's inequality give
\begin{align}
\|u\|_p^p \le & \|u \|_1^{p \beta} \left\|u\right\|_{\frac{(m_1+p-1)d}{d-2}}^{p(1-\beta)} \nonumber\\ 
\le & C(d) \left\|u\right\|_1^{p \beta} \left\|\nabla u\right\|_2^{\frac{2p(1-\beta)}{m_1+p-1}} \nonumber \\
\le & \frac{2 m_1 p(p-1)}{(m_1+p-1)^{2}} \int \left|\nabla u^{\frac{m_1+p-1}{2}}\right|^{2} dx +C\left(\|u_0\|_1 \right),                                                     
\end{align} 
where $\beta=\frac{\frac{1}{p}-\frac{d-2}{(m_1+p-1)d}}{1-\frac{d-2}{(m_1+p-1)d}} $, and we have used $\frac{2p(1-\beta)}{m_1+p-1}=\frac{2(p-1)}{p-1+m_1-1+2/d}<2$ for $m_1>1-2/d$. Similarly, using \eqref{qrange}, we arrive at
\begin{align}
\|v\|_q^q \le \frac{2 m_2 q(q-1)}{(m_2+q-1)^{2}} \int \left|\nabla v^{\frac{m_2+q-1}{2}}\right|^{2} dx +C\left(\|v_0\|_1 \right). 
\end{align}
Therefore, it follows from \eqref{pqtime} that
\begin{align}
\frac{d}{dt} \left(\int u^p dx+\int v^q dx\right) +\int u^p dx+\int v^q dx \le C,
\end{align}
which implies that for all $1\le p, q<\infty$,
\begin{align}
\|u(\cdot,t)\|_p \le C,\quad \|v(\cdot,t)\|_q \le C.
\end{align}

{\it\textbf{Step 2}} ($L^\infty$-estimates) \quad We apply the Moser iteration method to establish uniform boundedness of the weak solution, 
\begin{align}\label{Linftime}
\|u(\cdot,t)\|_{L^\infty(\R^d)}+\|v(\cdot,t)\|_{L^\infty(\R^d)} \le  C \left( \|u_0\|_{L^1\cap L^\infty(\R^d)},\|v_0\|_{L^1\cap L^\infty(\R^d)} \right),\quad \text{for any } t>0.
\end{align}
By Lemma \ref{ueps}, there exists a time $T_w>0$ and a solution $(u,v)$ to \eqref{uvsystem} on $[0,T_w)$ with initial data \eqref{initialdata}. The uniform bound \eqref{Linftime} ensures, again by Lemma \ref{ueps}, that the solution can be extended globally in time, i.e., $T_w=\infty$. Hence, we obtain a global weak solution $(u,v)$ to \eqref{uvsystem}. 
\end{proof}

The proof of Theorem \ref{main} is now completed by combining Proposition \ref{blowup} (finite-time blow-up) and Proposition \ref{globalexistence} (global existence).

\section{Conclusions} \label{sec5}

This paper investigates the two-species chemotaxis system \eqref{uvsystem} with diffusion exponents on the critical curve 
\begin{align}\label{m1m2}
\frac{1}{m_1}+\frac{1}{m_2}=\frac{d+2}{d}.
\end{align}
This choice ensures the invariance of the free energy $F(u,v)$ under the scaling \eqref{uvscale}, while preserving the $L^{m_1}$ norm of $u(x,t)$ and the $L^{m_2}$ norm of $v(x,t)$. Moreover, this curve marks the threshold that distinguishes stationary solutions that have compactly support from those that do not. Specifically, the stationary solutions $(U_s,V_s)(x)$ of \eqref{uvsystem} are radially symmetric and non-increasing (up to translation). Our main result establishes that the pair $\left(\|U_s\|_{L^{m_1}(\R^d)},\|V_s\|_{L^{m_2}(\R^d)}\right)$ provides a sharp criterion for initial data, separating solutions that exist for infinite time from those that blow up in finite time. Precisely, under the condition $F(u_0,v_0)<F(U_s,V_s)$, we prove that there exists a global weak solution when the initial data satisfies $\|u_0\|_{L^{m_1}(\R^d)}<\|U_s\|_{L^{m_1}(\R^d)}$ and $\|v_0\|_{L^{m_2}(\R^d)}<\|V_s\|_{L^{m_2}(\R^d)}$, whereas there are blowing-up solutions for $\|u_0\|_{L^{m_1}(\R^d)}>\|U_s\|_{L^{m_1}(\R^d)}$ and $\|v_0\|_{L^{m_2}(\R^d)}>\|V_s\|_{L^{m_2}(\R^d)}$. 

In conclusion, this work establishes a sharp dichotomy for the two-species chemotaxis system at the energy-critical exponent. We have proved that the scaling-invariant norms of the stationary solutions in $L^{m_1}(\R^d) \times L^{m_2}(\R^d)$ provide a critical threshold that sharply separates global existence from finite-time blow-up. This result underscores a fundamental phenomenon: the balance between degenerate diffusion and nonlocal aggregation is precisely quantified by these scaling-invariant norms. The proof relies on a novel combination of free energy dissipation, variational characterization of steady states and a detailed analysis of the Hardy-Littlewood-Sobolev inequality. 

\section*{Data Availability Statement}
All data generated or analysed during this study are included in this published article (and its supplementary information files). All data that support the findings of this study are included within the article (and any supplementary files).

\section*{Statements and Declarations}
This research did not receive any specific grant from funding agencies in the public, commercial, or not-for-profit sectors. The author has no competing interests to declare that are relevant to the content of this article.

%% 参考文献（按 JMPA 要求：数字顺序引用）

\end{document}